\documentclass[10 pt]{article}

\usepackage{graphicx,amsmath,latexsym,amssymb,amsthm,fancyhdr}

\font\chuto=cmbx10 at 16pt  

\newtheorem{thm}{Theorem}[section]

\theoremstyle{definition}

\newtheorem{rem}[thm]{Remark}

\numberwithin{equation}{section}

\begin{document}

\setlength{\unitlength}{1cm}

\centerline {\bf \chuto  Three  pearls   of    Bernoulli   numbers}
\medskip

\renewcommand{\thefootnote}{\fnsymbol{footnote}}

\vskip.8cm
\centerline {Abdelmoum\'ene Z\'ekiri$^{1,}${\footnote{\text{A. Zekiri :  czekiri@gmail.com}}}
              and Farid Bencherif$^{2,}${\footnote{\text{F. Bencherif :  fbencherif@yahoo.fr}}}
}

\renewcommand{\thefootnote}{\arabic{footnote}}

\vskip.5cm
\centerline{$^1$Faculty of Mathematics, USTHB, Algiers, Algeria
}
\centerline{$^2$Faculty of Mathematics, USTHB, Algiers, Algeria
}

\vskip .5cm

\vskip .5cm

\begin{abstract}
The Bernoulli numbers are fascinating and ubiquitous numbers; they occur in
several domains of Mathematics like Number theory ( FLT), Group theory,
Calculus and even in Physics. Since Bernoulli's work, they are yet studied
to understand their deep nature \cite{MAZ}, \cite{MIK} and particularly to
find relationships between them. In this paper, we give, firstly, a short
response \cite{ZEK-BEN2} to a problem stated, in 1971, by Carlitz \cite{CAR}
and studied by many authors like Prodinger \cite{PRO}; the second pearl is
an answer to a question raised, in 2008, by Tom Apostol \cite{APO}.The third
pearl is another proof of a relationship already given in 2011, by the
authors \cite{ZEK-BEN1}

\end{abstract}




\section{Introduction}

The aim of this work is to give original proofs of three relationships
involving Bernoulli numbers. In the fisrt section, we give a short proof to
a problem stated, in 1971, by Carlitz \cite{CAR} and studied by many authors
like Prodinger \cite{PRO}. In the second section, we give a response to a
question raised by Apostol in 2008 in his relevant paper \cite{APO}. In the
third section, we expose a different proof of a relationship already given
by us in 2011 \cite{ZEK-BEN2}. \vspace{0.5cm}

\section{Pearl \#1: Carlitz's Problem}

In Mathematics Magazine, Vol. 44, No. 2 (Mar., 1971), pp. 105-114+101,
Carlitz states the following problem :

\ \ \ \ \ \textit{define }$\{B_{n}\}$\textit{\ by means of }$B_{0}=1$\textit{%
\ and for }$n>1$\textit{\ }%
\begin{equation*}
\sum_{k=0}^{n}\binom{n}{k}B_{k}=B_{n}
\end{equation*}%
\textit{show that for arbitrary} $m,n>0$ 
\begin{equation*}
(-1)^{m}\sum_{k=0}^{m}\binom{m}{k}B_{n+k}=(-1)^{n}\sum_{k=0}^{n}\binom{n}{k}%
B_{m+k}
\end{equation*}%
This identity was firstly proved by Shanon \cite{SHA} in 1971, by Gessel 
\cite{GES} in 2003, by Wu, Sun and Pan \cite{WU-SUN-PAN} in 2004, by
Vassilev-Missana \cite{VAS-MIS} in 2005, by Chen and Sun \cite{CHE-SUN} in
2009, by Gould and J. Quaintance \cite{GOU-QUA} in 2014 and by Prodinger 
\cite{PRO} in 2014. The Prodinger's proof is very short and uses a two
variables formal series. In fact, one can see that Carlitz's problem can be
easily deduced from the following relationship already proved in 2012 by
Bench\'{e}rif and Garici in 2012 \cite{BEN-GAR} : 
\begin{equation*}
(-1)^{m}\sum_{k=0}^{m+q}\binom{m+q}{k}\binom{n+q+k}{q}B_{n+k}-(-1)^{n+q}%
\sum_{k=0}^{n+q}\binom{n+q}{k}\binom{m+q+k}{q}B_{m+k}=0.
\end{equation*}%
Hereafter, we give a proof different from that was given by Prodinger.

\begin{proof}
We consider the linear functional $L$ defined on $\mathbb{Q}[x]$ by $%
L(x^{n})=B_{n}$ for $n\geq 0$, which gives 
\begin{equation*}
L\left( \left( x+\frac{1}{2}\right) ^{2n+1}\right) =B_{2n+1}\left( \frac{1}{2%
}\right) =0,
\end{equation*}%
see \cite{APO}, p.182, then the polynomial defined by : 
\begin{equation*}
P(x)=(-1)^{m+q}x^{n+q}(1+x)^{m+q}-(-1)^{n}x^{m+q}(1+x)^{n+q}
\end{equation*}%
satisfies 
\begin{equation*}
P\left( -\frac{1}{2}+x\right) +(-1)^{q}P\left( -\frac{1}{2}-x\right) =0
\end{equation*}%
and 
\begin{equation*}
P^{(q)}(-\frac{1}{2}+x)+P^{(q)}(-\frac{1}{2}-x)=0
\end{equation*}%
Now, with use of the equality : 
\begin{equation*}
L\left( \left( x+\frac{1}{2}\right) ^{2n+1}\right) =B_{2n+1}\left( \frac{1}{2%
}\right) =0.
\end{equation*}%
And as $P^{(q)}$ is an even polynomial , we get : 
\begin{equation*}
L\left( P^{(q)}(x)\right) =0
\end{equation*}%
Thus 
\begin{equation*}
\frac{1}{q!}P^{(q)}(x)=(-1)^{m}\sum_{k=0}^{m+q}\binom{m+q}{k}\binom{n+q+k}{q}%
x^{n+k}-(-1)^{n+q}\sum_{k=0}^{n+q}\binom{n+q}{k}\binom{m+q+k}{q}x^{m+k}=0
\end{equation*}%
and finally: 
\begin{equation*}
(-1)^{m}\sum_{k=0}^{m+q}\binom{m+q}{k}\binom{n+q+k}{q}B_{n+k}-(-1)^{n+q}%
\sum_{k=0}^{n+q}\binom{n+q}{k}\binom{m+q+k}{q}B_{m+k}=0
\end{equation*}%
which yields the identity wanted by Carlitz, by taking $q=0.$
\end{proof}

\section{Pearl \#2: APOSTOL'S PROBLEM}

In his relevant paper published in 2008, Tom Apostol writes : \textit{we
leave it as a challenge to the reader to find another proof of (42) as a
direct consequence of (3) without the use of integration}. In Apostol's
paper, (42) denotes the relationship: 
\begin{equation*}
\sum_{k=0}^{n}\binom{n}{k}\frac{B_{k}}{(n+2-k)}=\frac{B_{n+1}}{n+1},\quad
n\geq 1
\end{equation*}%
and (3) is one of the \textit{six} definitions of the Bernoulli numbers he
recalls to show his relationship and which is : 
\begin{equation*}
B_{0}=1,\qquad \sum_{k=0}^{n-1}\binom{n}{k}B_{k}=0\quad \text{for}\ n\geq 2
\end{equation*}%
As he said it, Apostol uses \textit{integration method} to deduce his (42)-
numbered relation from the Bernoulli numbers's definition that he has
chosen. To take up the challenge he has launched, we expose a proof without
use of \textit{integration method}.

\begin{proof}
( Answer to Apostol's problem)

Let's define the sequence $\left( u_{n}\right) $ by: 
\begin{equation*}
u_{n}:=\sum_{k=0}^{n}\binom{n+1}{k}B_{k}
\end{equation*}%
We can see that $u_{n}=0$ for $n\geq 1$. Writing: 
\begin{equation*}
\binom{n}{k}\frac{1}{n+2-k}=\frac{1}{n+1}\binom{n+1}{k}-\frac{1}{(n+1)(n+2)}%
\binom{n+2}{k}
\end{equation*}%
we get : 
\begin{equation*}
\sum_{k=0}^{n}\binom{n}{k}\frac{B_{k}}{n+2-k}=\frac{1}{n+1}\sum_{k=0}^{n}%
\binom{n+1}{k}B_{k}-\frac{1}{(n+1)(n+2)}\sum_{k=0}^{n}\binom{n+2}{k}B_{k}
\end{equation*}%
which yields : 
\begin{equation*}
\sum_{k=0}^{n}\binom{n}{k}\frac{B_{k}}{n+2-k}=\frac{1}{n+1}u_{n}-\frac{1}{%
(n+1)(n+2)}\left( u_{n+1}-\binom{n+2}{n+1}B_{n+1}\right)
\end{equation*}%
As if $n\geq 1$, we have $u_{n}=u_{n+1}=0$ and as $\binom{n+2}{n+1}=n+2$, we
get : 
\begin{equation*}
\frac{-1}{(n+1)(n+2)}\left( -\binom{n+2}{n+1}B_{n+1}\right) =\frac{B_{n+1}}{%
n+1}.
\end{equation*}

which gives the asked relation :

\begin{equation*}
\sum_{k=0}^{n}\binom{n}{k}\frac{B_{k}}{(n+2-k)}=\frac{B_{n+1}}{n+1},\quad
n\geq 1
\end{equation*}

Finally, Apostol's relationship is proved without use of \textit{integration
method}.
\end{proof}

\section{Pearl \#3: New proof of a relationship}

In our paper \cite{ZEK-BEN2}, we proved the following relationship: 
\begin{equation*}
\sum_{k=0}^{n+q}\binom{n+q}{k}\left( \prod_{j=1}^{q}(n+k+j)\right) B_{n+k}=0
\end{equation*}%
where $q$ is an odd number. For this, we showed that the two well-suited
polynomials : 
\begin{equation*}
H_{n}(x)=\frac{1}{2}x^{n+q}(x-1)^{n+q}
\end{equation*}%
and 
\begin{equation*}
K_{n}(x)=\sum_{k=0}^{n+q}\frac{\varepsilon _{n+k}}{n+q+k+1}\binom{n+q}{k}%
B_{n+q+k+1}(x)-B_{n+q+k+1}
\end{equation*}%
Are equal, where $B_{n}(x)$ and $B_{n}$ are respectively the Bernoulli
polynomials and the Bernoulli numbers defined by the generating function: 
\begin{equation*}
\frac{x}{e^{x}-1}e^{xz}=\sum_{n=0}^{+\infty }B_{n}(x)\frac{x^{n}}{n!}
\end{equation*}%
knowing that $B_{n}=B_{n}(0)=B_{n}(1)$, $n\geq 2$; $B_{2n+1}=0$, $n\geq 1$,
see \cite{APO} , relations (11), (12) (13) and (15). Furthermore, we shall
use the well-known equalities: 
\begin{equation*}
B_{n}(x+1)-B_{n}(x)=nx^{n-1},\qquad B_{n}^{\prime }(x)=nB_{n-1}(x),\qquad
\int_{x}^{x+1}B_{n}(t)dt=x^{n},\quad n\geq 0
\end{equation*}%
for $n\geq 1$, see e.g. \cite{APO}, relations (14), (27) and (30).

Now, to give another proof of the relation already proved in \cite{ZEK-BEN2}%
, we consider the polynomials: 
\begin{equation*}
P_{n}(x):=\frac{1}{2}x^{n+1}(x-1)^{n+1}
\end{equation*}%
\begin{equation*}
K_{n}(x):=\sum_{k=0}^{n+1}\binom{n+1}{k}\frac{1-(-1)^{n+1-k}}{2}B_{n+1+k}(x)
\end{equation*}%
\begin{equation*}
H_{n}(x):=\frac{1}{2}(n+1)x^{n}(x-1)^{n}(2x-1)
\end{equation*}%
and the automorphism of the $\mathbb{Q}$-space vector $\mathbb{Q}[x]$
defined by $f\left( P(x)\right) =\int_{x}^{x+1}P(t)dt$.

First of all, let's prove the

\begin{thm}
The two polynomials $K_{n}(x)$ and $H_{n}(x)$ are equal, i.e. $%
K_{n}(x)=H_{n}(x)$.
\end{thm}

\begin{proof}
\begin{eqnarray*}
\int_{x}^{x+1}P_{n}^{\prime }(t)dt &=&P_{n}(x+1)-P_{n}(x) \\
&=&\frac{1}{2}x^{n+1}(x+1)^{n+1}-\frac{1}{2}x^{n+1}(x-1)^{n+1} \\
&=&\frac{1}{2}x^{n+1}\left( (x+1)^{n+1}-(x-1)^{n+1}\right) \\
&=&x^{n+1}\sum_{k=0}^{n+1}\binom{n+1}{k}\frac{(1-(-1)^{n+1-k}}{2}x^{k} \\
&=&\sum_{k=0}^{n+1}\binom{n+1}{k}\frac{1-(-1)^{n+1-k}}{2}x^{n+1+k} \\
&=&\sum_{k=0}^{n+1}\binom{n+1}{k}\frac{1-(-1)^{n+1-k}}{2}%
\int_{x}^{x+1}B_{n+1+k}(t)dt \\
&=&\int_{x}^{x+1}\left( \sum_{k=0}^{n+1}\binom{n+1}{k}\frac{1-(-1)^{n+1-k}}{2%
}B_{n+1+k}(t)dt\right)
\end{eqnarray*}%
Thus, we can see that 
\begin{equation*}
f\left( P_{n}^{\prime }(t)\right) =f\left( \sum_{k=0}^{n+1}\binom{n+1}{k}%
\frac{1-(-1)^{n+1-k}}{2}B_{n+1+k}(t)\right)
\end{equation*}%
As $f$ is bijective, we get : 
\begin{equation*}
P_{n}^{\prime }(t)=\sum_{k=0}^{n+1}\binom{n+1}{k}\frac{1-(-1)^{n+1-k}}{2}%
B_{n+1+k}(t)
\end{equation*}

i.e.%
\begin{equation*}
P_{n}^{\prime }(t)=K_{n}(t)
\end{equation*}

Let's compute $P_{n}^{\prime }(x)$ : 
\begin{eqnarray*}
P_{n}^{\prime }(x) &=&\frac{1}{2}\left( x^{2}-x)^{n+1}\right) ^{\prime } \\
&=&\frac{1}{2}(n+1)(x^{2}-x)^{n}(2x-1) \\
&=&\frac{1}{2}(n+1)x^{n}(x-1)^{n}(2x-1) \\
&=&H_{n}(x)
\end{eqnarray*}%
i.e. 
\begin{equation*}
K_{n}(x)=H_{n}(x).
\end{equation*}
\end{proof}

\begin{thm}
The following identity holds: 
\begin{equation*}
\sum_{k=0}^{n+q}\binom{n+q}{k} \left(\prod_{j=1}^q(n+k+j)\right) B_{n+k}=0
\end{equation*}
\end{thm}

\begin{proof}
To get this, let's replace $n$ by $n+q-1$, $q\geq 1$, $q$ odd. Then we
compute the coefficient of $x^{q}$ in the equality: $%
K_{n+q-1}(x)=H_{n+q-1}(x)$. The coefficient of $x^{q}$ in the polynomial $%
K_{n}(x)$ 
\begin{equation*}
K_{n}(x):=\sum_{k=0}^{n+q}\binom{n+q}{k}\frac{1-(-1)^{n+q-k}}{2}B_{n+1+k}(x)
\end{equation*}%
is : 
\begin{equation*}
C_{q}:=[x]^{q}B_{n+q-1}(x),
\end{equation*}%
so that $C_{q}$ has the value 
\begin{equation*}
C_{q}=\sum_{k=0}^{n+q}\binom{n+q}{k}\left( \prod_{j=1}^{q}(n+k+j)\right) 
\frac{1-(-1)^{n+q-k}}{2}B_{n+1+k}
\end{equation*}%
On the other hand, the coefficient of $x^{q}$ in $H_{n+q-1}(x)$ is : 
\begin{equation*}
\left\{ 
\begin{array}{ll}
0 & \text{if}\ n\geq 1 \\ 
(-1)^{q+1}q & \text{if}\ n=0%
\end{array}%
\right.
\end{equation*}%
As we have : 
\begin{equation*}
\frac{1+(-1)^{m}}{2}B_{m}=\left\{ 
\begin{array}{ll}
B_{m} & \text{if}\ m\neq 1 \\ 
0 & \text{if}\ m=1%
\end{array}%
\right.
\end{equation*}%
We get the aimed relationship.

\begin{rem}
I would like to dedicate this modest contribution to the memory of Tom Mike
Apostol who passed away on 8 May of this year 2016.(APO2)
\end{rem}
\end{proof}

\end{document}